\documentclass[10pt]{amsart}
\usepackage{amsmath,amscd}
\usepackage{amsbsy}
\usepackage{amssymb}
\usepackage{amscd,amsthm}
\usepackage[all,cmtip]{xy}

\newtheorem{thm}{Theorem}
\newtheorem{lem}[thm]{Lemma}
\newtheorem{prop}[thm]{Proposition}
\newtheorem{cor}[thm]{Corollary}

\newtheorem{rema}[thm]{Remark}

\begin{document}
\begin{center}
\huge{Ulrich bundles on some twisted flags}\\[1cm]
\end{center}

\begin{center}

\large{Sa$\mathrm{\check{s}}$a Novakovi$\mathrm{\acute{c}}$}\\[0,4cm]
{\small August 2018}\\[0,3cm]
\end{center}

\noindent{\small \textbf{Abstract}. 
In this note we prove that certain twisted flag varieties carry Ulrich bundles.\\

Let $X\subset \mathbb{P}^N$ be a projective variety of dimension $d$. An \emph{Ulrich bundle} on $X$ is a vector bundle $\mathcal{E}$ satisfying $H^i(X,\mathcal{E}(-l))=0$ for any $i\in\mathbb{Z}$ and $1\leq l\leq d$. This notion was introduced in \cite{ESW} where the authors ask whether every projective variety admits an Ulrich bundle. The answer is known in a few cases: curves and Veronese varieties \cite{ESW}, \cite{H}, complete intersections \cite{BHU1}, generic linear determinantal varieties \cite{BHU}, Segre varieties \cite{CMRP}, rational normal scrolls \cite{MR}, \cite{FAN}, Grassmannians \cite{CMR}, some flag varieties \cite{CMR}, \cite{CHW}, some isotropic Grassmannians \cite{FO}, $\mathrm{K}$3 surfaces \cite{AFO}, \cite{FA}, abelian surfaces \cite{BE}, Enriques surfaces \cite{BN}, ruled surfaces \cite{ACMR} and Brauer--Severi varieties \cite{N1}. We refer to \cite{B} for an introduction to this subject. In this short note we prove that the existence of Ulrich bundles on certain twisted flags of type $A_n, B_n, C_n$ and $D_n$ follows easily from known results for partial flags and descent theory.\\

We recall the definition and some facts on twisted flags and refer to \cite{MPW} for details. Let $G$ be a semisimple algebraic group over a field $k$ and $G_s=G\otimes_k k^{sep}$. For a parabolic subgroup $P$ of $G_s$, one has a homogeneous variety $G_s/P$. A \emph{twisted flag} is variety $X$ such that $X\otimes_k k^{sep}$ is $G_s$-isomorphic to $G_s/P$ for some $G$ and some parabolic $P$ in $G_s$. Any twisted flag is smooth, absolutely irreducible and reduced. An algebraic group $G'$ is called twisted form of $G$ iff $G'_s\simeq G_s$ iff $G'={_\gamma} G$ for some $\gamma\in Z^1(k,\mathrm{Aut}(G_s))$. The group $G'$ is called an \emph{inner form} of $G$ if there is a $\delta\in Z^1(k,\bar{G}(k^{sep}))$ with $G'={_\delta}G$. Here $\bar{G}=G/Z(G)$. For an arbitrary semisimple $G$ over $k$, there is a unique (up to isomorphism) split semisimple group $G^d$ such that $G_s\simeq G_s^d$. If $G$ is an inner form of $G^d$, then $G$ is said to be of \emph{inner type}. For instance, let $A$ be a central simple algebra over $k$ of degree $n$ and $G=\mathrm{PGL}_1(A)$, then $G_s\simeq \mathrm{PGL}_n$ over $k^{sep}$. Hence $G$ is an inner form of $\mathrm{PGL}_n$. Since $\mathrm{PGL}_n$ is split, $G=\mathrm{PGL}_1(A)$ is of inner type. The inner twisted forms arising from $G=\mathrm{PGL}_1(A)$ can be described very explicitly (see \cite{MPW}, Section 5). One of these inner twisted forms is the \emph{generalized Brauer--Severi variety}. 
So let $m\leq n$. The generalized Brauer--Severi variety $\mathrm{BS}(m,A)$ is defined to be the subset of $\mathrm{Grass}_k(mn,A)$ consisting of those subspaces of $A$ which are right ideals of dimension $m\cdot n$ (see \cite{KNU} or \cite{BLA}). Recall that $\mathrm{Grass}_k(mn,A)$ is given the structure of a projective variety via the Pl\"ucker embedding (see \cite{BLA})
\begin{eqnarray*}
\mathrm{Grass}_k(mn,A)\longrightarrow \mathbb{P}(\wedge^{mn}(A)).
\end{eqnarray*}
This gives an embedding $\mathrm{BS}(m,A)\rightarrow \mathbb{P}(\wedge^{mn}(A))$ and a very ample line bundle $\mathcal{M}$.
Note that for any $\mathrm{BS}(m,A)$ there exists a finite Galois field extension $E$ such that $\mathrm{BS}(m,A)\otimes_k E\simeq \mathrm{Grass}_E(mn,n^2)\simeq \mathrm{Grass}_E(m,n)$. The Picard group $\mathrm{Pic}(\mathrm{Grass}_E(m,n))$ is isomorphic to $\mathbb{Z}$ and has ample generator $\mathcal{O}(1)\simeq \mathrm{det}(\mathcal{Q})$ with $\mathcal{Q}$ being the universal quotient bundle on $\mathrm{Grass}_E(m,n)$. Recall that $\mathrm{Pic}(\mathrm{BS}(m,A))\simeq\mathbb{Z}$ and that there is a positive generator $\mathcal{L}$ such that $\mathcal{L}\otimes_k E\simeq \mathcal{O}(r)$ for a suitable $r>0$. Since $\mathrm{Pic}(\mathrm{BS}(m,A))$ is cyclic, we have $\mathcal{L}^{\otimes s}\simeq \mathcal{M}$ for a suitable $s>0$. Therefore, $\mathcal{L}$ is ample. From the definition of $\mathrm{BS}(m,A)$ it is clear that $\mathcal{L}$ is also very ample. 
\begin{prop}
Let $X=\mathrm{BS}(m,A)$ be a genaralized Brauer--Severi variety over a field $k$ of characteristic zero and denote by $\mathcal{L}$ the very ample generator of $\mathrm{Pic}(X)$. Then $(X, \mathcal{L}^{\otimes d})$ carries an Ulrich bundle for all $d\geq 1$. 
\end{prop}
\begin{proof}
There is a finite Galois field extension $E$ of $k$ such that $X\otimes_k E$ is isomorphic to the Grassmannian $\mathrm{Grass}_E(mn,n^2)$. Consider the projection map $\pi\colon X\otimes_k E\rightarrow X$. Note that $\pi$ is finite and surjective. Let $k^{sep}$ be a separable closure containing $E$. According to \cite{CMR}, Theorem 3.6 there is an Ulrich bundle $\mathcal{F}$ on $X\otimes_k k^{sep}\simeq \mathrm{Grass}_{k^{sep}}(mn,n^2)$ if it is embedded via the Pl\"ucker embedding, or equivalently, if it is embedded via $\mathcal{O}_{\mathrm{Grass}_{k^{sep}}(mn,n^2)}(1)$. This Ulrich bundle is obtained in the following way: Let 
\begin{eqnarray*}
0\longrightarrow \mathcal{S}^{\vee}\longrightarrow \mathcal{O}_{\mathrm{Grass}_{k^{sep}}(mn,n^2)}^{\oplus n^2}\longrightarrow \mathcal{Q}\longrightarrow 0
\end{eqnarray*}
be the universal exact sequence. Then $\mathcal{F}$ is given by $\Sigma^{\lambda}(\mathcal{S}^{\vee})\otimes \Sigma^{\beta}\mathcal{Q}$ for suitable partitions $\lambda$ and $\beta$. Since $\mathcal{S}^{\vee}$ and $\mathcal{Q}$ are also defined on $X\otimes_k E\simeq \mathrm{Grass}_E(mn,n^2)=:\mathbb{G}$ and since the Schur functor $\Sigma$ commutes with base change, the vector bundle $\mathcal{F}$ is also defined over $\mathbb{G}$. In other words, there is a bundle $\mathcal{F}'$ on $\mathbb{G}$ such that $\mathcal{F}'\otimes_E k^{sep}\simeq \mathcal{F}$. If we embed $\mathbb{G}$ via the Pl\"ucker embedding, i.e via $\mathcal{O}_{\mathbb{G}}(1)$, we get $H^i(X,\mathcal{F}'(-l))=0$ for any $i\in\mathbb{Z}$ and $1\leq l\leq \mathrm{dim}(X)$. In fact, this follows from base change to $k^{sep}$. Therefore, $\mathcal{F}'$ is an Ulrich bundle for $(\mathbb{G},\mathcal{O}_{\mathbb{G}}(1))$. Now we use \cite{B}, Corollary 3.2 to see that $(\mathbb{G},\mathcal{O}_{\mathbb{G}}(d))$ carries an Ulrich bundle for all $d>1$. Note that there is a line bundle $\mathcal{O}_{\mathbb{G}}(r)$ on $\mathbb{G}$ such that $\mathcal{L}\otimes_k E\simeq \mathcal{O}_{\mathbb{G}}(r)$. Since there is an Ulrich bundle $\mathcal{E}$ for $(\mathbb{G},\mathcal{O}_{\mathbb{G}}(rd))$ for all $d\geq 1$, we can apply \cite{B}, (3.6) to conclude that $\pi_*\mathcal{E}$ is an Ulrich bundle for $(X, \mathcal{L}^{\otimes d})$. 
\end{proof}
\begin{rema}
\textnormal{The twisted Grassmannian $\mathrm{BS}(m,A)$ embedded into projective space by the very ample generator $\mathcal{L}$ cannot carry an Ulrich line bundle. Recall that $k^{sep}$ splits $A$, i.e $A\otimes_k k^{sep}\simeq M_n(k^{sep})$. In fact, if $\mathcal{G}$ is an Ulrich line bundle, the base change $\mathcal{G}\otimes_k k^{sep}$ would give an Ulrich line bundle on $\mathbb{G}:=\mathrm{Grass}_{k^{sep}}(mn,n^2)$ with respect to $\mathcal{L}\otimes_k k^{sep}\simeq \mathcal{O}_{\mathbb{G}}(r)$ for a suitable $r>1$. From \cite{CHW}, Proposition 2.1 we conclude that the degree of $\mathbb{G}$ is $>1$. Since $\mathrm{Pic}(\mathbb{G})$ is generated by $\mathcal{O}_{\mathbb{G}}(1)$, there cannot exist an Ulrich bundle of rank one (see \cite{B}, Section 4 for an explanation). So, as in the case of ordinary Brauer--Severi varieties, there are no Ulrich line bundles on generalized ones.}
\end{rema}
\begin{rema}
\textnormal{Proposition 1 from above concerns the existence of Ulrich bundles on $\mathrm{BS}(m,A)$. It would be interesting to relate the minimal rank of such a bundle to the invariants period and index of $A$. In the case of ordinary Brauer--Severi varieties, i.e $m=1$ this is done in \cite{N1}. Moreover, we wonder whether the obtained Ulrich bundles from Proposition 1 remain Ulrich in positive characteristic.}
\end{rema}
In general, the inner twisted flags arising from $G=\mathrm{PGL_1}(A)$, where $A$ is a central simple algebra of degree $n$, are varieties denoted by $\mathrm{BS}(n_1,...,n_l,A)$, $n_1<\cdots<n_l<n$, satisfying $\mathrm{BS}(n_1,...,n_l,A)\otimes_k k^s\simeq \mathrm{Flag}_{k^s}(n_1,...,n_l;n)$. These partial twisted flags parametrize sequences $I_1\subseteq\cdots \subset I_l\subseteq A$ of right ideals with $\mathrm{dim}(I_j)=n\cdot n_j$, for $j=1,...,l$ (see \cite{MPW}, Section 5). If $E$ is a splitting field of $A$, i.e $A\otimes_k E\simeq M_n(E)$, one has $\mathrm{BS}(n_1,...,n_l,A)\otimes_k E\simeq \mathrm{Flag}_E(n_1,...,n_l;n)$. For details, we refer to \cite{KNU} and \cite{MPW}. Recall that a flag $\mathbb{F}_L:=\mathrm{Flag}_L(n_1,...,n_l;n)$ over a field $L$ has $l$ projections $p_i\colon \mathrm{Flag}_L(n_1,...,n_l;n)\longrightarrow \mathrm{Grass}_L(l_i,n)$. The Picard group of $\mathbb{F}_L$ is generated by $\mathcal{L}_i=p_i^*\mathcal{O}_{\mathrm{Grass}_{L}(l_i,n)}(1)$. A line bundle $\mathcal{R}$ on $\mathbb{F}_L$ is ample if and only if $\mathcal{R}=\mathcal{L}_1^{\otimes a_1}\otimes\cdots\otimes\mathcal{L}_l^{\otimes a_l}$ with $a_i>0$. We set $\mathcal{O}_{\mathbb{F}_L}(1)=\mathcal{L}_1\otimes\cdots\otimes\mathcal{L}_l$.  
\begin{lem}
Let $X=\mathrm{BS}(n_1,...,n_l,A)$ be a twisted flag as above and $E$ a finite splitting field of $A$. Then there is a very ample line bundle $\mathcal{L}$ on $X$ such that $\mathcal{L}\otimes_k E\simeq \mathcal{O}_{\mathbb{F}_E}(r)$ for a suitable $r>0$.
\end{lem}
\begin{proof}
We give an elementary proof. A more structural argument is given in Proposition 7 below. The twisted flag $X$ has $l$ projections to generalized Brauer--Severi varieties
\begin{eqnarray*}
\pi_i\colon \mathrm{BS}(n_1,...,n_l,A)\longrightarrow \mathrm{BS}(l_i,A).
\end{eqnarray*}
Let us denote by $\mathcal{O}_{\mathrm{BS}(l_i,A)}(1)$ the ample generators of $\mathrm{Pic}(\mathrm{BS}(l_i,A))$. The line bundles $\mathcal{M}_i=\pi_i^*\mathcal{O}_{\mathrm{BS}(l_i,A)}(1)$ are in $\mathrm{Pic}(X)$ and so are the bundles 
\begin{eqnarray*}
\mathcal{M}_1^{\otimes a_1}\otimes\cdots \otimes \mathcal{M}_l^{\otimes a_l}
\end{eqnarray*}
for $a_i\in\mathbb{Z}$. Note that after base change to the splitting field $E$, we have $\mathcal{M}_i\otimes_k E\simeq \mathcal{L}_i^{\otimes b_i}$ for suitable positive $b_i\in\mathbb{Z}$. Therefore, we get 
\begin{eqnarray*}
(\mathcal{M}_1\otimes\cdots \otimes \mathcal{M}_l)\otimes_k E\simeq \mathcal{L}_1^{\otimes b_1}\otimes\cdots\otimes\mathcal{L}_l^{\otimes b_l}.
\end{eqnarray*}
Let $m$ be the least common multiple of all the $b_i$. By definition, there are positive integers $c_i$ such that $c_i\cdot b_i=m$. We set $\mathcal{N}=\mathcal{M}_1^{\otimes c_1}\otimes\cdots \otimes \mathcal{M}_l^{\otimes c_l}$ and obtain
\begin{eqnarray*}
\mathcal{N}\otimes_k E\simeq \mathcal{L}_1^{\otimes b_1\cdot c_1}\otimes\cdots\otimes\mathcal{L}_l^{\otimes b_l\cdot c_l}\simeq \mathcal{O}_{\mathbb{F}_E}(m).
\end{eqnarray*}
Since $\mathcal{O}_{\mathbb{F}_E}(m)$ is ample on $\mathbb{F}_E$, we get that $\mathcal{N}$ must be ample on $\mathrm{BS}(n_1,...,n_l,A)$. We take $\mathcal{L}$ to be $\mathcal{N}^{\otimes s}$ for a suitable large integer $s>0$. In fact, there are plenty of such very ample line bundles. This completes the proof.
\end{proof}
\begin{prop}
Let $A$ be a degree $n$ central simple algebra over a field $k$ of characteristic zero and $E$ a finite splitting field. Let $X$ be one of the following twisted flag varieties:
\begin{eqnarray*}
\mathrm{BS}(1,n-1,A),\;\; \mathrm{BS}(1,n-2,A),\;\; \mathrm{BS}(2,n-2,A),\;\; \mathrm{BS}(m,m+1,A),\;\; \mathrm{BS}(m,m+2,A).
\end{eqnarray*}
Then there are very ample line bundles $\mathcal{L}$ on $X$ satisfying $\mathcal{L}\otimes_k E\simeq \mathcal{O}_{\mathbb{F}_E}(r)$ for suitable $r>0$. Moreover, for every such very ample line bundle $(X,\mathcal{L}^{\otimes d})$ carries an Ulrich bundle for all $d\geq 1$.
\end{prop}
\begin{proof}
The existence of very ample line bundles $\mathcal{L}$ satisfying the desired property follows from Lemma 4. After base change to some splitting field $L$ of $A$, the variety $X$ becomes one of the following flags: 
\begin{eqnarray*}
\mathrm{Flag}_L(1,n-1;n),\;\; \mathrm{Flag}_L(1,n-2;n),\;\; \mathrm{Flag}_L(2,n-2;n),\;\; \mathrm{Flag}_L(m,m+1;n),\\
\mathrm{Flag}_L(m,m+2;n).
\end{eqnarray*}
We consider the projection map $\pi\colon X\otimes_k E\rightarrow X$ which is finite and surjective. Let $k^{sep}$ be a separable closure containing $E$. Note that $k^{sep}$ is also a splitting field of $A$. According to \cite{CHW}, there is an Ulrich bundle on $X\otimes_k k^{sep}$ with respect to $\mathcal{O}_{X\otimes_k k^{sep}}(1)$. This Ulrich bundle is obtained in the following way: Let $0=\mathcal{T}_0\subset \mathcal{T}_1\subset \mathcal{T}_2\subset \mathcal{T}_3$ be the collection of tautological subbundles on $X\otimes_k k^{sep}$ and $\mathcal{U}_i=\mathcal{T}_i/\mathcal{T}_{i-1}$. Then $\Sigma^{\lambda}(\mathcal{U}_1^{\vee})\otimes\Sigma^{\alpha}(\mathcal{U}_2^{\vee})\otimes \Sigma^{\beta}(\mathcal{U}_3^{\vee})$ is an Ulrich bundle for suitable partitions $\lambda,\alpha$ and $\beta$. These bundles can also be defined over $\mathbb{F}_E=X\otimes_k E$. We argue as in the proof of Proposition 1 to conclude that there is an Ulrich bundle on $\mathbb{F}_E$ with respect to $\mathcal{O}_{\mathbb{F}_E}(1)$. Again, we use \cite{B}, Corollary 3.2 to see that $(\mathbb{F}_E,\mathcal{O}_{\mathbb{F}_E}(d))$ carries an Ulrich bundle for all $d>1$. Now take a very ample line bundle $\mathcal{L}$ on $\mathbb{F}_E$ satisfying $\mathcal{L}\otimes_k E\simeq \mathcal{O}_{\mathbb{F}_E}(r)$ for some $r>0$. Since there is an Ulrich bundle $\mathcal{E}$ for $(\mathbb{F},\mathcal{O}_{\mathbb{F}_E}(rd))$ for all $d\geq 0$, we can apply \cite{B}, (3.6) to conclude that $\pi_*\mathcal{E}$ is an Ulrich bundle for $(X, \mathcal{L}^{\otimes d})$. This proves the assertion. 
\end{proof}

\begin{thm}[Type $A_n$]
Let $A$ be a degree $n$ central simple algebra over a field $k$ of characteristic zero and $E$ a finite splitting field. Let $X$ one of the following twisted flags:
\begin{eqnarray*} 
\mathrm{BS}(m,A),\;\; \mathrm{BS}(1,n-1,A),\;\; \mathrm{BS}(1,n-2,A),\;\; \mathrm{BS}(2,n-2,A),\;\; \mathrm{BS}(m,m+1,A),\\
\mathrm{BS}(m,m+2,A).
\end{eqnarray*}
Then there are very ample line bundles $\mathcal{L}$ on $X$ satisfying $\mathcal{L}\otimes_k E\simeq \mathcal{O}_{\mathbb{F}_E}(r)$ for suitable $r>0$. Moreover, for every such very ample bundle $(X,\mathcal{L}^{\otimes d})$ carries an Ulrich bundle for all $d\geq 1$.
\end{thm}
\begin{prop}
Let $X$ be a smooth projective geometrically integral variety over a field $k$. Assume $X\otimes_k k^{sep}$ is embedded into projective space via $\mathcal{O}_{X\otimes_k k^{sep}}(1)$, i.e $\mathcal{O}_{X\otimes_k k^{sep}}(1)=i^*\mathcal{O}_{\mathbb{P}^N}(1)$ for an embedding $i\colon X\otimes_k k^{sep}\rightarrow \mathbb{P}^N$. Then there is a very ample line bundle $\mathcal{L}$ on $X$ satisfying $\mathcal{L}\otimes_k k^{sep}\simeq \mathcal{O}_{X\otimes_k k^{sep}}(r)$ for a suitable $r>0$.  
\end{prop}
\begin{proof}
Let $G$ be the absolute Galois group. It is well know that there is an exact sequence arising from the Leray spectral sequence 
\begin{eqnarray*}
0\longrightarrow \mathrm{Pic}(X)\longrightarrow \mathrm{Pic}(X\otimes_k k^{sep})^G\stackrel{\delta}\longrightarrow \mathrm{Br}(k)\longrightarrow \mathrm{Br}(k(X)).
\end{eqnarray*}
Note that every element of $\mathrm{Br}(k)$ has finite order. Now let $\mathcal{O}_{X\otimes_k k^{sep}}(m)$ with $m>0$ be an element of the cokernel of $\mathrm{Pic}(X)\rightarrow \mathrm{Pic}(X\otimes_k k^{sep})^G$. If this element is trivial, we are done. Assume $\mathcal{O}_{X\otimes_k k^{sep}}(m)$ is non-trivial. Then $\delta(\mathcal{O}_{X\otimes_k k^{sep}}(m))$ is a non-trivial Brauer-equivalence class $[B]\in\mathrm{Br}(k)$. If $d>0$ is the order of $[B]$ in $\mathrm{Br}(k)$, we obtain that $\delta(\mathcal{O}_{X\otimes_k k^{sep}}(md))=[k]$. This implies that there exists a line bundle $\mathcal{L}$ on $X$ such that $\mathcal{L}\otimes_k k^{sep}\simeq\mathcal{O}_{X\otimes_k k^{sep}}(md)$. Since $\mathcal{O}_{X\otimes_k k^{sep}}(md)$ is very ample, we conclude that $\mathcal{L}$ must be very ample. Therefore, there is a very ample line bundle $\mathcal{L}$ on $X$ satisfying $\mathcal{L}\otimes_k k^{sep}\simeq \mathcal{O}_{X\otimes_k k^{sep}}(r)$ for a suitable $r>0$.
\end{proof}
Let $Y$ be one of the following flags: symplectic Grassmannians $\mathrm{IGrass}_{k^{sep}}(2,2n)$ for $n\geq 2$, odd and even-dimensional quadrics, orthogonal Grassmannians $\mathrm{OGras}_{k^{sep}}(2,m)$ for $m\geq 4$, $\mathrm{OGrass}_{k^{sep}}(3,4q+6)$ for $q\geq 0$ and $\mathrm{OGrass}_{k^{sep}}(4,8)$. Then $\mathrm{Pic}(Y)\simeq \mathbb{Z}$ is generated by a very ample line bundle (see \cite{FO}, Section 2). We will denote this very ample generator by $\mathcal{O}_Y(1)$. For inner twisted forms of $Y$, we refer to \cite{MPW}, Section 5.
\begin{thm}[Type $B_n, C_n$ and $D_n$]
Let $X$ be an inner twisted form of $Y$ from above. Then there are very ample line bundles $\mathcal{L}$ on $X$ satisfying $\mathcal{L}\otimes_k k^{sep}\simeq \mathcal{O}_{Y}(r)$ for suitable $r>0$. Moreover, for every such very ample bundle $(X,\mathcal{L}^{\otimes d})$ carries an Ulrich bundle for all $d\geq 1$.
\end{thm}
\begin{proof}
The existence of very ample line bundles $\mathcal{L}$ satisfying the desired property follows from Proposition 7. Note that for every inner twisted flag from above, there is a finite field extension $E$ of $k$ such that $X\otimes_k E$ is isomorphic to the corresponding flag. We prove the statement only for the case $X=\mathrm{IGrass}_{k^{sep}}(2,2n)$ for $n\geq 2$, because the proof for the remaining cases is analogous. So let $E$ be a finite splitting field contained in $k^{sep}$, i.e $X\otimes_k E\simeq \mathrm{IGrass}_E(2,2n)$. The main theorem of \cite{FO} states that the listed flags carry Ulrich bundles over $k^{sep}$ with respect to $\mathcal{O}_Y(1)$. In particular, $Y:=\mathrm{IGrass}_{k^{sep}}(2,2n)$ carries an Ulrich bundle $\mathcal{E}$ with respect to $\mathcal{O}_Y(1)$ which is given as $\Sigma^{(2n-3,0,...,0)}(\mathcal{U}^{\vee})$ with $\mathcal{U}$ being the tautological subbundle (see \cite{FO}, Corollary 3.6). This vector bundle can also defined over $\mathrm{IGrass}_E(2,2n)$, i.e there is a vector bundle $\mathcal{E'}$ on $\mathrm{IGrass}_E(2,2n)$ such that $\mathcal{E}'\otimes_E k^{sep}\simeq \mathcal{E}$. We argue as in Proposition 1 to conclude that $\mathcal{E'}$ is an Ulrich bundle on $\mathrm{IGrass}_E(2,2n)$ with respect to $\mathcal{O}_{\mathrm{IGrass}_E(2,2n)}(1)$. We use \cite{B}, Corollary 3.2 to see that $(\mathrm{IGrass}_E(2,2n),\mathcal{O}_{\mathrm{IGrass}_E(2,2n)}(d))$ carries an Ulrich bundle for all $d>1$. From Proposition 7 we can conclude that there is a very ample line bundle $\mathcal{L}$ on $X$ satisfying $\mathcal{L}\otimes_k E\simeq \mathcal{O}_{\mathrm{IGrass}_E(2,2n)}(r)$ for a suitable $r>0$. As in Propositions 1 and 5 we use the projection $\pi\colon X\otimes_k E\rightarrow X$, which is finite and surjective, to produce an Ulrich bundle for $(X,\mathcal{L}^{\otimes d})$ for all $d\geq 1$.
\end{proof}
Let $Y$ be one of the flags from above. Then $\mathrm{Pic}(Y)\simeq \mathbb{Z}$ is generated by the very ample line bundle $\mathcal{O}_Y(1)$. For inner twisted forms $X$ of $Y$, this implies $\mathrm{Pic}(X)\simeq \mathbb{Z}$. The ample generator $\mathcal{L}$ of $\mathrm{Pic}(X)$ obviously satisfies $\mathcal{L}\otimes_k k^{sep}\simeq \mathcal{O}_Y(r)$ for a suitable $r>0$. Moreover, the ample generator of $\mathrm{Pic}(X)$ is also very ample. 
\begin{cor}
Let $X$ be as in Theorem 8 and denote by $\mathcal{L}$ the very ample generator of $\mathrm{Pic}(X)$. Then $(X,\mathcal{L}^{\otimes d})$ carries an Ulrich bundle for all $d\geq 1$. 
\end{cor}  
\noindent{\small \textbf{Acknowledgement}. I thank Nikita Karpenko, Daniel Krashen and Kirill Zainoulline for answering questions concerning twisted flags. This research was conducted in the framework of the research training group GRK 2240: Algebro-geometric Methods in Algebra, Arithmetic and Topology, which is funded by the $\mathrm{DFG}$.



{\small MATHEMATISCHES INSTITUT, HEINRICH--HEINE--UNIVERSIT\"AT 40225 D\"USSELDORF, GERMANY}\\
E-mail adress: novakovic@math.uni-duesseldorf.de

\end{document}